\documentclass[12pt,a4paper,
fleqn,reqno]{amsart}                 
\usepackage[T1]{fontenc}                   
\usepackage[utf8]{inputenc}                 
\usepackage[english]{babel}    
 \usepackage{csquotes}
\usepackage{graphicx}                      %
\usepackage[font=small]{quoting}    
\usepackage{float}
\usepackage{caption}  
\usepackage[top=1.2in,bottom=1.8in,left=1.0in,right=1.0in]{geometry}                  
\usepackage{verbatim}
\usepackage{color}
\usepackage{xcolor}
\usepackage{picture}
\usepackage{amsmath,amsthm,amsfonts,amssymb}
\usepackage{comment}
\usepackage{mathtools}

\usepackage{hyperref}
\hypersetup{colorlinks,linkcolor={blue}, citecolor={blue},urlcolor={black}}  
\usepackage{enumitem}
\usepackage[url=false,doi=false,isbn=false, giveninits=true,style=numeric,maxbibnames=99]{biblatex}
\newtheorem{thm}{Theorem}[section] 
\newtheorem{lem}[thm]{Lemma} 
\newtheorem{cor}[thm]{Corollary} 
\newtheorem{prop}[thm]{Proposition} 
 
\newtheorem{defn}[thm]{Definition} 
\theoremstyle{definition} 
\newtheorem{rem}[thm]{Remark} 

\theoremstyle{remark}

\numberwithin{equation}{section}

\def\S{\Sigma} 
\def\n{\nabla}

\def\p{\partial}

\def\a{\alpha}
\def\b{\beta}
\def\n{\nabla}

\def\p{\partial}

\def\a{\alpha}
\def\b{\beta}
\def\g{\gamma}

\def\l{\lambda}
\def\s{\sigma}

\def\ov{\overline}

\def\n{\nabla}
\def\<{\langle}
\def\>{\rangle}

\def\n{\nabla}

\def\NN{\mathbb{N}}

\def\RR{\mathbb{R}}

\def\SS{\mathbb{S}}

\def\p{\partial}

\def\a{\alpha}
\def\b{\beta}
\def\g{\gamma}

\def\l{\lambda}
\def\K{\mathcal{K}}
\def\s{\sigma}

\def\ov{\overline}

\def\wh{\widehat}
\def\S{\Sigma}

\def\V{\mathcal V}

\def\C{\mathcal{C}}

\def\ol{\overline}
\def\wh{\widehat}
\def\wt{\widetilde}

\def\bt{\mathbf{b}_\theta}

{\left\{\begin{array}{@{}l@{}}}{\end{array}\right.}
\patchcmd{\abstract}{\scshape\abstractname}{\textbf{\abstractname}}{}{}
\makeatletter 
\def\@makefnmark{} 
\makeatother

\numberwithin{equation}{section}
\numberwithin{exa}{section}

\begin{document}
\title[Alexandrov Fenchel inequalities II]
{Alexandrov-Fenchel inequalities \\for convex hypersurfaces in the half-space \\with capillary boundary II}
    
\author[X. Mei]{Xinqun Mei}
\address[X. Mei]{Key Laboratory of Pure and Applied Mathematics, School of Mathematical Sciences, Peking University,  Beijing, 100871, P.R. China}
\email{\href{qunmath@pku.edu.cn}{qunmath@pku.edu.cn}}

\author[G. Wang]{Guofang Wang}
\address[G. Wang]{Mathematisches Institut, Albert-Ludwigs-Universit\"{a}t Freiburg, Freiburg im Breisgau, 79104, Germany}
\email{\href{mailto:guofang.wang@math.uni-freiburg.de} {guofang.wang@math.uni-freiburg.de}}

\author[L. Weng]{Liangjun Weng}
\address[L. Weng]{Centro di Ricerca Matematica Ennio De Giorgi, Scuola Normale Superiore, Pisa, 56126, Italy \& Dipartimento di Matematica, Università di Pisa, Pisa, 56127, Italy.}
\email{\href{mailto:liangjun.weng@sns.it}  {liangjun.weng@sns.it}}

\author[C. Xia]{Chao Xia}
\address[C. Xia]{School of Mathematical Sciences, Xiamen University, Xiamen, 361005, P.R. China}
\email{\href{mailto:chaoxia@xmu.edu.cn} {chaoxia@xmu.edu.cn}}

	

\begin{abstract} In this paper, we provide an affirmative answer to \cite[Conjecture 1.5]{WWX2023} on the Alexandrov-Fenchel inequality for quermassintegrals for convex capillary hypersurfaces in the Euclidean half-space. More generally, we establish a theory for capillary convex bodies in the half-space and prove a general Alexandrov-Fenchel inequality for mixed volumes of capillary convex bodies. The conjecture \cite[Conjecture 1.5]{WWX2023} follows as its consequence.
	\end{abstract}

 \maketitle	
 \section{Introduction}
      Let $\RR^{n+1}_+=\{x\in \RR^{n+1}\,|\, x_{n+1}>0\}$ be the upper half-space. Let $\S\subset \overline{\RR^{n+1}_+}$ be a properly embedded, smooth compact hypersurface with boundary in $\overline{\RR^{n+1}_+}$ such that $\text{int}(\S)\subset \RR^{n+1}_{+}$ {and} $\p \Sigma\subset \p \RR^{n+1}_{+}.$
      We call $\S\subset \ol{\RR^{n+1}_+}$  a capillary hypersurface if $\S$ intersects $\p\RR^{n+1}_+$ at a constant contact angle $\theta\in (0, \pi)$.
      Let $\wh{\S}$ be the bounded closed region in $\ol{\RR^{n+1}_+}$ enclosed by $\S$ and $\p\RR^{n+1}_+$. The boundary of $\widehat{\S}$  consists of two parts: one is $\Sigma$ and  the other, which
will be denoted by $\widehat{\p \Sigma}$,  lies on  $\p \RR^{n+1}_+$. 
In \cite{WWX2023}, for a capillary hypersurface $\S\subset \ol{\RR^{n+1}_+}$, we introduced a new family of geometric functionals as follows:
\begin{eqnarray*} 
	\V_{0,\theta}(\widehat{\Sigma})\coloneqq |\widehat{ \S}|,
	\qquad \V_{1,\theta} (\widehat{\Sigma})\coloneqq \frac{1}{n+1} \left ( |\S|-\cos\theta |\widehat{\p\S}| \right),\end{eqnarray*}
and for $1\leq k\leq n$,
\begin{eqnarray}\label{quermassintegrals}
	\V_{k+1,\theta}(\widehat{\S})\coloneqq
	\frac{1}{n+1}\left(\int_\S H_kdA - \frac{\cos\theta \sin^k\theta  }{n}\int_{\p\S} H_{k-1}^{\p\S}ds \right),
\end{eqnarray}where  $H_k$ is the normalized $k$-th mean curvature of $\S\subset \ol{ \RR^{n+1}_+}$ and $H_{k-1}^{\p\S}$ is the normalized  $(k-1)$-th mean curvature of $\p\S\subset \widehat{\p \Sigma}\subset\p \RR^{n+1}_+$ (cf. \cite[Section 2.2]{WWX2023}). We use the convention that $H_0\equiv 1$ and $H_0^{\p\S}\equiv 1$. 
$\V_{0,\theta}$ is the volume of $\widehat \S$ and $\V_{1,\theta}$ is the capillary area of $\Sigma$, up to a multiple constant. $\V_{k,\theta}$ ($2\le k\le n+1)$  are natural capillary counterparts of quermassintegrals for closed hypersurfaces. (For quermassintegrals we refer to \cite{Sch} and for capillary hypersurfaces to \cite{Finn} and also \cite{Maggi}.) Especially,	as shown in \cite[Theorem 2.6]{WWX2023}, the first variation of $\V_{k,\theta}$ gives rise to $H_k$.  Moreover, the following conjecture was proposed in \cite[Conjecture 1.5]{WWX2023}.

\medskip

\noindent{\bf Conjecture.}  {\it 
	Let $\Sigma\subset    \ol{{\mathbb{R}}^{n+1}_+}$ $(n\geq 2)$ be a convex capillary hypersurface with a contact angle $\theta \in (0, {{\pi}})$. Then there holds
	\begin{eqnarray}\label{af ineq}
	\frac{ \V_{k,\theta}(\widehat{\S}) }{\bt}
	\geq 
	\left( \frac{  \V_{l,\theta}(\widehat{\S}) }{ \bt} \right)^{\frac{n+1-k}{n+1-l}} ,  \qquad  \forall ~ 0\le l<k\le n,
	\end{eqnarray}
	with equality holds if and only if $\Sigma$ is a capillary spherical cap. 
}

\medskip

Here $(n+1) \bt$ is the {capillary} area of $\C_{\theta,1}$, the capillary spherical cap with radius $1$. One can check that $\bt$ is the volume of $\widehat{\C_{\theta, 1}}$. 
The family of capillary spherical caps lying entirely in $\ol{{\mathbb{R}}^{n+1}_+}$ and intersecting $\partial{\RR}^{n+1}_{+}$ with a constant contact angle $\theta\in (0,\pi)$ is given by
\begin{eqnarray}\label{sph-cap}
\C_{\theta, r}\coloneqq\left\{\xi \in \ol{{\mathbb{R}}^{n+1}_+} \big|~ |\xi-r\cos\theta e|=r \right\}, \qquad r\in [0, \infty),
\end{eqnarray}
which is clearly a portion of a sphere of radius $r$ and centered at $r\cos\theta e$. Here we denote by $e\coloneqq-E_{n+1}=(0,\cdots, 0, -1)$ the unit outward normal of $\partial{\RR}^{n+1}_{+}\subset\ol{{\mathbb{R}}^{n+1}_+}$. For simplicity, we denote $\C_{\theta }\coloneqq\C_{\theta, 1}$. In the following, we denote $\s$ and $d\s$ the round metric and its associated volume form on $\C_{\theta}$ respectively.

For $k=1$ and $l=0$, \eqref{af ineq} is just the relative (or capillary) isoperimetric inequality which holds for any hypersurfaces in $\ol{{\mathbb{R}}^{n+1}_+}$ (see e.g. \cite[Theorem 19.21]{Maggi}). Therefore, \eqref{af ineq} are generalizations of the capillary isoperimetric inequality or capillary counterparts of the Alexandrov-Fenchel inequalities for closed convex hypersurfaces
(see e.g. \cite{BZ} or \cite{Sch}). 

In \cite{WWX2023}, we established \eqref{af ineq} for $k=n$ and any $0\leq l< n$, in the case $\theta\leq {\pi}\slash {2}$,  by using a locally constrained inverse curvature flow. Subsequently, Hu-Wei-Yang-Zhou \cite{HWYZ} used a similar locally constrained inverse type curvature flow to establish \eqref{af ineq} for all $0\leq l<k\leq n$, again in the case $\theta\leq {\pi}\slash {2}$. The new key ingredient in \cite{HWYZ} is to verify the convexity preservation along the flow via a boundary tensor maximum principle. One can also refer to \cite{MWW, MW24} for using various locally constrained curvature-type flows to prove partial results of \eqref{af ineq}  for the case $l=0$. 
The contact angle range restriction  $\theta\leq  \pi \slash 2$ is very crucial in the aforementioned papers
in order to obtain  $C^2$ estimates. It actually becomes a challenging problem how to obtain $C^2$ estimates for those curvature flows in the case $\theta >{\pi}\slash  2$.
Recently, Wang-Weng-Xia proved the Minkowski-type inequality, i.e., $k=2$  and $l=1$  in \eqref{af ineq}, for all $\theta\in(0,\pi)$ in \cite{WWX2022}, by assuming only the mean convexity and the star-shapedness of the hypersurface. 
In \cite{WWX2022}, the authors studied an inverse mean curvature type flow and explored the specific
divergence structure of the mean curvature to directly obtain Schauder’s estimate for the
radial function for star-shaped hypersurfaces. This method is not easy to be generalized for  $k>2$.
Nevertheless, this result makes us believe that \eqref{af ineq} is true for the whole range of $\theta$.
The objective of this paper is to confirm this belief and fully resolve the {Conjecture} for all $\theta\in(0,\pi)$ using a different approach.

 In the classical theory of convex bodies, the isoperimetric-type inequalities for quermassintegrals state that for a convex body (a compact convex set with non-empty interior) $K\subset\mathbb{R}^{n+1}$ with $C^2$ boundary,  there holds
	\begin{eqnarray}\label{af ineq-x}
	\frac{ W_{k}(K)}{{\bf b}_{n+1}}
	\geq 
	\left( \frac{  W_{l}(K) }{{\bf b}_{n+1}} \right)^{\frac{n+1-k}{n+1-l}} ,  \qquad  \forall ~ 0\le l<k\le n,
	\end{eqnarray}
	where ${\bf b}_{n+1}=|\mathbb{B}^{n+1}|$, the volume of unit ball $\mathbb{B}^{n+1}$, and  $W_k(K)$ is the $k$-th quermassintegral defined by $$W_k(K)=\frac{1}{n+1}\int_{\p K}H_{k-1} dA.$$ 
 Moreover, equality in \eqref{af ineq-x} holds if and only if $K$ is a round ball. It is in fact a special case of a more general classical Alexandrov-Fenchel inequality for mixed volumes, which says that for a family of convex bodies $K_i\subset\mathbb{R}^{n+1}$, $i= 1,\cdots, n+1$, there holds
\begin{eqnarray}\label{classical-af}
V^2(K_1, K_2, K_3\cdots, K_{n+1})\ge V(K_1, K_1, K_3\cdots, K_{n+1}) V(K_2, K_2, K_3\cdots, K_{n+1}),
\end{eqnarray}
	where $V(K_1, K_2\cdots, K_{n+1})$ is the so-called mixed volume. 
	When $K_i, $ $i=1,\cdots, n+1$, has  $C^2$ boundary and $h_{i}:\SS^n\to \mathbb{R}$ is the support function of $K_i$, then it is known that
	$$V(K_1, K_2\cdots, K_{n+1})=V(h_1, h_2\cdots, h_{n+1})=\int_{\SS^n} h_1 Q(A[h_2], \cdots, A[h_{n+1}])d\s,$$
where $A[h]=\n^2 h+h \s$ and $Q$ is the so-called mixed discriminant. In this case, equality in \eqref{classical-af} holds if and only if $K_1$ and $K_2$ are homothetic. For the Alexandrov-Fenchel inequalities for closed convex hypersurfaces, see e.g. \cite[Section 20]{BZ} and \cite[Chapter 7]{Sch}. Recently, Guan-Ma-Trudinger-Zhu established in \cite{Guan} a form of Alexandrov-Fenchel inequality for a larger class of closed hypersurfaces, $k^*$-convex hypersurfaces.

In this paper, we shall make a systematic study of a special class of convex bodies in ${\mathbb{R}}^{n+1}$, that is, capillary convex bodies in $\ol{{\mathbb{R}}^{n+1}_+}$. By a capillary convex body $\widehat{\S}\subset \ol{{\mathbb{R}}^{n+1}_+}$, we mean that $\widehat{\S}$ is a convex body for a capillary hypersurface $\S$. Denote by $\mathcal{K}_\theta$ the family of capillary convex bodies in $\ol{{\mathbb{R}}^{n+1}_+}$. 
For $\widehat{\S}\in\mathcal{K}_\theta$, let $\nu$ be the outer unit normal of $\S$, i.e., the Gauss map of $\S$. We
define the {\it capillary Gauss map} by $\tilde{\nu}=\nu+\cos\theta e$. It turns out  $\tilde{\nu}: \S\to \C_\theta$ is a diffeomorphism and one can reparametrize $\S$ by its inverse on $\C_\theta$. Moreover, there is a one-to-one correspondence between a capillary convex body and  a   capillary support 
function $h\in C^2(\C_\theta)$, which is a capillary convex function, that is, $A[h]>0$ and 
\begin{eqnarray}\label{boundary-cond}
\n_\mu h=\cot\theta h ~~ \hbox{ on }\p \C_\theta,
\end{eqnarray}
where $\mu$ is the outward conormal of $\p \C_\theta\subset \C_\theta$.
See Proposition \ref{equ pro} below or \cite[Section 2]{MWW3}. Since $\widehat{\S}$ is not $C^2$ at the points in $\p \S$, the classical support function of $\widehat{\S}$, which is defined on whole $\SS^n$, is not $C^2$. Therefore, we use instead the capillary support function which is only defined on $\C_\theta$ and satisfies the boundary condition \eqref{boundary-cond}.
For example, the capillary support function of $\widehat{\C_\theta}$ is given by (see Section \ref{sec2.1})
\begin{eqnarray*}
\ell(\xi)=\sin^{2}\theta+\cos\theta \<\xi, e\>, \quad \xi\in \C_\theta.
\end{eqnarray*}
This is the main point of view different from the one in the classical case, where a convex body is determined by a (smooth or non-smooth) support function defined on the whole unit sphere.

For a family of capillary convex bodies $\{\widehat{\S}_i\}_{i=1}^{m}\subset \mathcal{K}_\theta$, whose capillary support functions are given by $h_{i}$, the mixed volume
$V(\widehat{\S}_{i_1}, \cdots, \widehat{\S}_{i_{n+1}})$ is defined in the classical way.
We will see that the Minkowski sum
$K=\sum\limits_{i=1}^{m} \l_i \widehat{\S}_i$ lies in $\mathcal{K}_\theta$ and its capillary support function is given by $h_K= \sum\limits_{i=1}^{m} \l_i h_{i}.$
Using $$(n+1)|K|=\int_{\p K\cap \RR^{n+1}_+} \<X,\nu\>dA=\int_{\C_\theta} h_K\det (A[h_K]) d\s,$$
one can check that the mixed volume is given by
	\begin{eqnarray*}\label{defn-V}
			V(\widehat{\S}_{i_1}, \cdots, \widehat{\S}_{i_{n+1}})=V(h_{i_1}, \cdots, h_{i_{n+1}})\coloneqq\frac{1}{n+1}\int_{\C_{\theta}}h_{i_1}Q\left(A[h_{i_2}], \cdots, A[h_{i_{n+1}}]\right)d\s.
		\end{eqnarray*}
		We first observe that the quermassintegrals introduced in \cite{WWX2023} can be reinterpreted as a special case of the above mixed volume. Precisely, 
			\begin {eqnarray*}
			 \mathcal{V}_{n-k+1,\theta}(\widehat{\Sigma})=V(\underbrace{\widehat{\S}, \cdots, \widehat{\S}}_{k ~\rm{ copies}} ,\widehat{\C_\theta},\cdots, \widehat{\C_\theta})= V(\underbrace{h, \cdots, h}_{k ~\rm{ copies}} ,\ell,\cdots, \ell), ~~~ 0\leq k\leq n+1,
		\end{eqnarray*}	where $h$ and $\ell$ are the capillary support functions of $\widehat{\Sigma}$ and $\widehat{\C_\theta}$ respectively. 
By virtue of this observation,  the conjectured inequality \eqref{af ineq} follows in fact directly from the classical Alexandrov-Fenchel inequality \eqref{classical-af}. It remains to consider the equality case. The characterization of the equality case of the Alexandrov-Fenchel inequality is actually a long-standing open problem, see e.g. \cite[Section 7.6]{Sch}. Very recently, Shenfeld-van Handel made a big breakthrough in \cite{SvH3} and classified the extremals of the Alexandrov-Fenchel inequality for convex polytopes. They actually solved the problem where the convex bodies $K_3, K_4, \cdots, K_{n+1}$ in \eqref{classical-af} are a combination of polytopes, zonoids, and smooth bodies, see \cite[Section 14]{SvH3}. The case we study here is not directly included. In order to study the mixed volume for the class $\mathcal{K}_\theta$ of capillary convex bodies, we will use capillary support functions defined only on $\C_\theta$. We prove the following capillary version of the Alexandrov-Fenchel inequality with rigidity characterization.
\begin{thm}\label{thm1}
	Let $\widehat{\S}_i\in\mathcal{K}_{\theta}$ for~$1\leq i\leq n+1$.
	 Then 
\begin{eqnarray}\label{general AF}
		V^{2}(\widehat{\S}_{1}, \widehat{\S}_{2}, \wh{\S}_3,\cdots, \widehat{\S}_{n+1})\geq V(\widehat{\S}_{1}, \widehat{\S}_{1},\wh{\S}_3,\cdots, \widehat{\S}_{n+1})V(\widehat{\S}_{2}, \widehat{\S}_{2},\wh{\S}_3,\cdots, \widehat{\S}_{n+1}).
	\end{eqnarray}
	Equality holds if and only if 
 the capillary support functions $h_j$  of $\widehat{\S}_j, j=1, 2$, satisfy $$h_{1}=a h_{2}+\sum_{i=1}^n a_i\<\cdot, E_i\>, ~~\text{ in } ~~ \C_\theta$$  for some constants $a, a_{i} \in \RR$, $i=1,\cdots, n$, and $\{E_i\}_{i=1}^n$  the horizontal coordinate unit vectors of $\ol{\RR^{n+1}_+}$.
\end{thm}

As a consequence of Theorem \ref{thm1}, we confirm the above \textbf{Conjecture}.
 \begin{thm}\label{thm2}
	For any convex capillary hypersurface $\Sigma\subset\ol{{\RR}^{n+1}_{+}}$  with a contact angle $\theta\in (0, \pi)$, \eqref{af ineq} is true, 
 with equality  if and only if $\Sigma$ is a spherical cap.\end{thm}

The proof of Theorem \ref{thm1} is inspired by a previous paper of Shenfeld-van Handel \cite{Shenfeld}, where a new proof of the classical Alexandrov-Fenchel inequality is provided by using its closed relationship with a spectral problem on $\SS^{n}$ and its associated Bochner formula.  For our situation, we establish a similar relationship between the capillary Alexandrov-Fenchel inequality and a spectral problem on $\C_\theta$ with a boundary condition. Precisely, the Alexandrov-Fenchel inequality \eqref{function AF} for the mixed volume of capillary convex bodies can be translated into a sharp lower bound for a spectral problem on $\C_\theta$ associated with certain elliptic operator $\mathcal{A}$ with a Robin boundary condition, \eqref{eigenvalue problem}. For our operator, we proved that it is self-adjoint and the corresponding positive eigenspace has dimension one (cf. \eqref{eigenvalue problem} below), which in turn justifies the validity of the Alexandrov-Fenchel inequality \eqref{function AF}. 
The self-adjointness follows from the capillarity, namely, the intersection angle $\theta$ is constant. 

We also recommend  \cite{SV22} for a further application of this idea, which provides a new tool for the study of mixed volumes in the theory of convex bodies. 

We emphasize that the capillary Gauss map $\tilde \nu:\Sigma \to \C_\theta$ and the capillary support function  $h:\C_\theta \to \RR$ play a crucial role in this paper. They were already crucially used in our previous work on studying the capillary Minkowski problem \cite{MWW3}.

\ 

\textbf{The paper is organized as follows}: In Section \ref{sec2},  we provide basic properties about the capillary convex body and its associated capillary convex function,  the mixed volume of capillary convex bodies and its relation with the quermassintegrals. In Section \ref{sec3}, we prove  Theorem \ref{thm1} on the Alexandrov-Fenchel inequalities for mixed volumes and thus confirm the  {Conjecture} as a special case of Theorem \ref{thm1}. 

\

\section{Capillary convex bodies and mixed volume}\label{sec2}

	\subsection{Capillary convex bodies and capillary convex functions}\label{sec2.1}
\begin{defn}
	For $\theta\in (0,\pi)$, we say that $\S$ is a capillary hypersurface (with constant contact angle $\theta$) in $\overline{\RR^{n+1}_+}$ if
\begin{eqnarray}\label{capillary condition}
   \<\nu, e\>=-\cos \theta,\quad \text{along}~\partial \S.
\end{eqnarray} 
We call $\widehat{\S}$ a capillary convex body if $\widehat{\S}$ is a convex body (a compact convex set with non-empty interior) in $\RR^{n+1}$ and $\S$ is a capillary hypersurface in $\ol{\RR^{n+1}_+}$.
The class of capillary convex bodies in $\overline{\RR^{n+1}_+}$ is denoted by $\mathcal {K}_\theta$.\end{defn}

The simplest capillary convex body is given by $\widehat{\C}_{\theta}\in \K_\theta$. Recall that $\C_\theta$ was defined by \eqref{sph-cap} above.  Next for any $\widehat{\Sigma}\in \mathcal {K}_\theta$ we give a parametrization of $\S$   via the (capillary) Gauss map $\tilde \nu$.
\begin{lem}\label{gaussmap}The Gauss map $\nu: \S\to \SS^n$ has its image in 
	\begin{eqnarray*}
		\SS^n_\theta \coloneqq\{ y\in \SS^n \,|\,  y_{n+1} \geq \cos \theta\},
	\end{eqnarray*}
	and $\nu: \S\to \SS^n_\theta$ is a diffeomorphism.
	\end{lem}

 \begin{proof} Since $\S$ is a smooth,  strictly convex capillary hypersurface, by \cite[Corollary~2.5]{WWX2023}, we see that $\partial\S\subset \partial{\RR^{n+1}_{+}}$ is a  strictly convex, closed  hypersurface. In particular, $\p \S$ is connected and topologically an $n-1$-sphere.  By \cite[Corollary 3.1 and Notes 3.2]{Ghomi}, we see that $\nu$ is one-to-one on $\p \S$ and furthermore, $\nu: \S\to \nu(\S)\subset \SS^{n}$ is a diffeomorphism. On the other hand, it follows from \eqref{capillary condition} that $\nu(\partial \S)=\p\SS_\theta^{n}$. Also, at least $E_{n+1}\in \SS^n_\theta$ lies in $\nu(\S)$. It follows that $\nu(\S)=\SS_\theta^{n}$ and the assertion follows.
\end{proof}
 
Let $T:\SS^n_\theta\to\C_\theta$ be the translation given by $$T(y)= y+\cos\theta e.$$ Instead of using the usual Gauss map $\nu$, it is more convenient to use the following map
	\begin{eqnarray}\label{cap Gauss map}
	\tilde \nu\coloneqq T \circ \nu: \S \to \C_\theta,\end{eqnarray}
which we call {\it capillary Gauss map} of $\Sigma$.
 
From Lemma \ref{gaussmap}, we can parametrize $\S$ by the inverse capillary Gauss map, i.e., $X:\C_{\theta}\to \S$,  given by
	\begin{eqnarray*}
		X(\xi)= \tilde \nu ^{-1} (\xi) =\nu^{-1}\circ T^{-1}(\xi)=\nu^{-1}(\xi-\cos\theta e).
	\end{eqnarray*}
		The support function of $\Sigma$ is given by 
	\begin{eqnarray*}
		h(X)=\<X,\nu(X)\>.
	\end{eqnarray*}
	By the above parametrization, $h$ is regarded as a function on $\C_\theta$, 
	\begin{eqnarray}\label{support}
		h(\xi) =\<X(\xi),\nu(X(\xi ))\>=\<X(\xi), T^{-1} (\xi )\>=
		\< \tilde \nu ^{-1}(\xi)  , \xi -\cos\theta e\>.
	\end{eqnarray}
{To distinguish with the ordinary support function of the convex body $\widehat{\S}$, we call $h$ in \eqref{support} {\it capillary support function} of $\S$ (or of $\widehat{\S}$). 
 
 \begin{rem} In \cite{MWW3}, we called $u(\xi)\coloneqq{h(\xi)} \slash (\sin^2 \theta + \cos \theta \<\xi, e\>)$ for $\xi\in \C_\theta$, as the capillary support function of $\S$. It is easy to see that $\n_\mu u=0$ on $\p\C_\theta$. $u$ has a clear geometric meaning (see \cite[Remark 2.3]{MWW3}) and is more suitable for this name. 
 \end{rem}

It is clear that  the capillary Gauss map for $ \C_\theta$  is the identity map 
from $\C_\theta \to \C_\theta$ and	\begin{eqnarray*}
		h(\xi)=\<\xi  , \xi -\cos\theta e\>=
		\sin ^2\theta + \cos\theta\< \xi, e\>.
	\end{eqnarray*}
	For simplicity, we denote 
 \begin{eqnarray}\label{ell}
     \ell(\xi)\coloneqq\sin^{2}\theta+\cos\theta \<\xi, e\>.
 \end{eqnarray} Therefore $\ell$  is  the capillary support function of $\C_{\theta}$ (or of $\wh{\C_\theta}$) in the latter context.

Denote $D$ the Euclidean covariant derivative, $\s$ and $\nabla$  the spherical metric and the covariant derivative on $\C_{\theta}$, then we have the following Lemma (see also \cite[Section 2]{MWW3}).
	\begin{lem}\label{para-lem} For the parametrization  $X:\C_{\theta}\to \S$, 
	\begin{itemize}
	\item[(1)] $X(\xi)=\nabla h(\xi) +h(\xi)T^{-1}(\xi)$.
		 \item[(2)]  $\n_\mu h=\cot \theta h$ along  $\p \C_{\theta}$, where $\mu$ is the outward unit conormal to $\p \C_{\theta}\subset \C_{\theta}$.
	 \item[(3)] The principal radii of $\S$ at $X(\xi)$ are given by the eigenvalues of $(\n^2 h+ h\s)$  with respect to the metric $\s$. In particular, $(\n^2 h+ h\s)>0$ on $\C_{\theta}$.
	 \end{itemize}	\end{lem}
\begin{proof}
Let $\{e_i\}_{i=1}^n$ be an orthonormal frame on $\C_{\theta}$. From \eqref{support}, we see
\begin{eqnarray*}
\n_{e_i} h(\xi)=\<D_{e_i} X, \xi -\cos\theta e\>+\<X, e_i\>=\<X, e_i\>, ~~~~1\leq i\leq n.
\end{eqnarray*}
Here the second equality follows from the fact that $D_{e_i} X$ is tangential to $\S$ while $\xi -\cos\theta e$ is normal to $\S$. The first assertion follows.

Using (1), we get for $\xi\in \p \C_{\theta}$, \begin{eqnarray*}
\n_\mu h(\xi)=\<\mu(\xi), X(\xi)\>,
\end{eqnarray*}
since $\<\mu(\xi), T^{-1}(\xi)\>=0$. On the other hand, by the capillary boundary condition, 
\begin{eqnarray*}
			\<\mu(\xi), X(\xi)\>&=& \<\frac{1}{\sin\theta}(e+\cos\theta\nu(X(\xi))),X(\xi)\>=\cot\theta\<\nu(X(\xi)), X(\xi)\>=\cot\theta h(\xi),		\end{eqnarray*}
since $\<X(\xi), e\>=0$ when $\xi\in \p \C_{\theta}$. The second assertion follows.

Next, using the Gauss formula $$D_{e_j}D_{e_i} X=\hbox{ tangential part} -B_{ij}\nu,$$ where $B_{ij}$ is the second fundamental form of $\S$, we have
\begin{eqnarray*}
\n_{e_j}\n_{e_i} h(\xi)&=&\<D_{e_j}D_{e_i} X, \xi -\cos\theta e\>+\<\n_{e_i}X, D_{e_j}\xi\> +\<\n_{e_j}X, D_{e_i}\xi\>+\<X,  \n_{e_j}\n_{e_i}\xi \>
\\&=&\<D_{e_j}D_{e_i} X, \xi -\cos\theta e\>-\<X,  \n_{e_j}\n_{e_i}\xi \>
\\&=&\<-B_{ij}(\xi -\cos\theta e), \xi -\cos\theta e\>- \<X, \s_{ij}(\xi -\cos\theta e)\>
\\&=&-B_{ij}-\s_{ij}h(\xi).
\end{eqnarray*}

Hence the second fundamental form $(B_{ij})$ is given by $(\n^2 h+ h\s)$.
By the Weingarten formula $D_{e_i}\nu=g^{kl}B_{ik}D_{e_l}X$, we see
\begin{eqnarray*}
\s_{ij}&=&\<D_{e_i}\xi, D_{e_j}\xi\>=\<D_{e_i}(\xi -\cos\theta e), D_{e_j}(\xi -\cos\theta e)\>
\\&=&\<g^{kl}B_{ik}D_{e_l}X, g^{pq}B_{jp}D_{e_q}X\>=B_{ik}B_{jp}g^{kp}.
\end{eqnarray*}
Hence \begin{eqnarray*}
g_{kp}=\s^{ij}B_{ik}B_{jp}.
\end{eqnarray*}
It follows that the principal radii, that are the eigenvalues of $(g_{kp}B^{jp})$, are given by the eigenvalues of $(\n^2 h+ h\s)$ with respect to the metric $\s$. The third assertion follows. 

\end{proof}

	\begin{defn}\label{defn-convex capillary fun}
		For $\theta\in (0, \pi)$, $f\in  C^{2}(\C_{\theta})$  is called a capillary function if it satisfies the following Robin-type boundary condition
		\begin{eqnarray}
			\label{robin}
			\nabla_\mu f=\cot \theta f\quad \text{on}\quad \partial \C_{\theta},
		\end{eqnarray}	
		If in addition, $(\n^{2}f+f\sigma)>0$ on $\C_{\theta}$, $f$ is called a capillary convex function on $\C_{\theta}$.
	\end{defn}
The next proposition says that a capillary convex function on $\C_{\theta}$  yields a capillary convex body in $\overline{\RR^{n+1}_+}$, and vice versa. This can be compared to the one-to-one correspondence of convex bodies in ${\RR^{n+1}}$ and convex functions on $\SS^n$.

\begin{prop}\label{equ pro}
		Let $h\in C^{2}(\C_{\theta})$. Then $h$ is a capillary convex function if and only if $h$ is the capillary support function of a capillary convex body $\wh{\S}\in \K_\theta$. 
	\end{prop}
\begin{proof}
We have already seen from Lemma \ref{para-lem} that
if $h\in C^{2}(\C_{\theta})$ is the capillary support function of a convex capillary hypersurface $\Sigma\subset \overline{\RR^{n+1}_+}$, then  $h$ is a  capillary convex function. 
It remains to show the converse.  

Let $h\in C^{2}(\C_{\theta})$ be a  capillary convex function. Consider the hypersurface $\S$ given by the map
\begin{eqnarray*}
X: &&\C_{\theta}\to\RR^{n+1},\\
	&&X(\xi)=\n h(\xi)+h(\xi) T^{-1}(\xi).
		\end{eqnarray*}
		
		For  $\xi\in\partial\C_{\theta}$, it is easy to see $\<\xi, e\>=0$ and $\<\mu(\xi), e\>=\sin\theta$.
Then by  \eqref{robin}, we have 
		\begin{eqnarray*}
			\<X(\xi), e\>=\<\n h(\xi)+h(\xi)T^{-1}(\xi), e\>
			= \n_{\mu}h\cdot  \< \mu, e\>+h(\xi)\<\xi-\cos\theta e, e\>=0.
		\end{eqnarray*}
		which implies $X(\p \C_{\theta})=\partial\Sigma\subset \partial {\mathbb{R}}^{n+1}_{+}$.
		
For  any tangent vector field $V$ on $\C_{\theta}$, we have
\begin{eqnarray}\label{derivative}
			D_{V}X(\xi)&=&D_{V}(\n h+h T^{-1}(\xi))\nonumber
			\\&=& \n_{V}(\n h) -\<V, \n h\> T^{-1}(\xi) + D_{V} h T^{-1}(\xi)+ h(\xi)V\nonumber\\&=&(\n^2 h+h\s)|_{\xi}(V).		\end{eqnarray}
	Since  $(\n^2 h+h\s)>0$, we see $X$ is an immersion.
Also from \eqref{derivative}, we see that for an orthonormal frame $\{e_i\}_{i=1}^n$  at $\xi$,
$$T_{X(\xi)}\S={\rm span}\{D_{e_i}X(\xi)\}_{i=1}^n ={\rm span}\{e_i\}_{i=1}^n .$$ 
It follows that 
$\nu(X(\xi))=T^{-1}(\xi)=\xi-\cos\theta e$ is the unit normal to $\S$ at $X(\xi)$. 

On the other hand, for $\xi\in {\rm int}(\C_{\theta})$, we claim that $\<X(\xi), E_{n+1}\>>0$ for each $\xi\in {\rm int}(\C_{\theta})$. To see this,	 for each $\xi\in  {\rm int}(\C_{\theta})$, let $\xi_0\in \p \C_\theta$ be the nearest point  in $\p \C_\theta$ to $\xi$ and $\g: [0, 1]\to \C_\theta$ be the geodesic from $\xi_0$ to $\xi$.  Then
\begin{eqnarray*}
			\<X(\xi), E_{n+1}\>&=&\<X(\xi), E_{n+1}\>-\<X(\xi_0), E_{n+1}\>
			\\&=&\int_0^1\<D_{\dot{\g}(t)}X(\g(t)), E_{n+1}\>dt
	\\&=&\int_0^1 (\n^2 h+h\s)|_{\g(t)}(\dot{\g}(t), \dot{\g}(t) )\< E_{n+1}, \dot{\g}(t)\>dt.
		\end{eqnarray*}
		The last equality holds since for $V\in T_{\g(t)} \C_\theta$ such that $\<V, \dot{\g}(t)\>=0$, it holds $\< E_{n+1}, V\>=0$.
		Since $(\n^2 h+h\s)>0$ and $\< E_{n+1}, \dot{\g}(t)\>>0$, we obtain that $\<X(\xi), E_{n+1}\>>0$. 
		Hence $X(\C_{\theta})\subset \overline{\RR^{n+1}_+}$.
				
Next, we prove $X(\C_{\theta})=\S\subset \overline{\RR^{n+1}_+}$ is (locally) strictly convex and capillary.
First, in particular, 
$$\<\nu(X(\xi)), e\>=\<T^{-1}(\xi), e\>=-\cos\theta\quad \hbox{for}~\xi\in \p \C_{\theta},$$
which implies it is capillary.
Second by \eqref{derivative}, $$D_{e_i}\nu=D_{e_i}(\xi-\cos\theta e)=(\n^2 h+h\s)^{-1} \cdot D_{e_j}X,$$
Hence the eigenvalues of $(\n^2 h+h\s)^{-1}|_{\xi}$ give the principal curvatures of $\S$ at $X(\xi)$. Thus $X: \C_{\theta} \to \overline{\RR^{n+1}_+}$ is strictly convex.

Finally, we prove that $X: \C_{\theta}\to \overline{\RR^{n+1}_+}$ is an embedding. For this, we look at the Gauss map $\nu: \S\to \SS^n$. From the strictly convexity and capillarity of $\S$, we know from \cite[Corollary 3.1]{Ghomi} that
$\nu$ is one-to-one on $\p \S$ and furthermore, $\nu$ is a diffeomorphism on $\S$ (as in the proof of Lemma \ref{gaussmap}). On the other hand, we know that 
$\nu(\S)=\SS^n_{\theta}= \C_{\theta}-\cos\theta e$. Hence $X(\C_\theta)=\S=\nu^{-1}(\C_{\theta}-\cos\theta e)$ is an embedding.

This completes the proof of the lemma.

\end{proof}

Next, we derive a simple and crucial fact about the capillary functions.	
	\begin{lem}\label{pro-f}
		If $f$ is a capillary function on $\C_{\theta}$, then for all $1\leq \alpha\leq n-1$,  
		\begin{eqnarray}\label{conormal direction mu}
			\n^{2}f(e_{\alpha}, \mu )=0, \quad \text{on}\quad \partial\C_{\theta},
		\end{eqnarray}
	where $\{e_\a\}_{1\le \a\le n-1}$ is an orthonormal frame of $\p\C_{\theta }$.  
	\end{lem}
	\begin{proof}
		Note that $\partial\C_{\theta}$ is a totally umbilical hypersurface in $\C_{\theta}$ with principal curvature $\cot\theta$, we have $\n_{e_{\alpha}} \mu=\cot\theta e_{\alpha}$ on $\p\C_\theta$. Together with \eqref{robin}, for all $1\leq \a \leq n-1$,
		\begin{eqnarray*}
			\n^{2}f(e_{\alpha}, \mu)=\n_{e_{\alpha}}(\n_{\mu}f)-\n_{\n_{e_{\alpha}} \mu} f=0.    \end{eqnarray*}We complete the proof.
	\end{proof}

\subsection{Mixed volume for capillary convex bodies}\label{sec2.2}
	
\begin{prop}\label{mink-sum}Let $K_i\in \mathcal{K}_\theta$, $i=1,\cdots, m$. Then the Minkowski sum $K\coloneqq\sum\limits_{i=1}^m \l_iK_i\in \mathcal{K}_\theta$, where  $\sum\limits_{i=1}^{m}\lambda_{i}>0$ and $\l_1,\ldots, \l_m\geq 0$. Moreover, \begin{eqnarray}\label{sum support}
    h_K=\sum\limits_{i=1}^m \l_i h_{K_i},
\end{eqnarray} where $h_K$ and $h_{K_i}$ are the capillary support functions of $K$ and $K_i$ respectively.
\end{prop} 

\begin{proof} 
Let $\wt{h}_{K_i}\coloneqq\widehat{h}_{K_i}(\cdot-\cos\theta e)$ with $\widehat{h}_{K_i}$ being the classical support function of $K_i$, viewed as a function on $\SS^n+\cos\theta e$ for $1\leq i\leq m$. Then $\wt{h}_K=\sum\limits_{i=1}^m \l_i \wt{h}_{K_i}$ is the support function of $K$, which is defined on $\SS^n+\cos\theta e$.
On the other hand, denote
$h_{K_i}\in C^2(\C_{\theta})$ the capillary support function of $K_i\in \K_\theta$, then 
$h_{K_i}=\wt{h}_{K_i} \big|_{\C_\theta}$.
Define \begin{eqnarray*}\label{sum support_1}
    h\coloneqq\sum_{i=1}^m \l_i h_{K_i}= \sum_{i=1}^m \l_i \wt{h}_{K_i}|_{\C_\theta}=\wt{h}_K|_{\C_\theta}.
\end{eqnarray*} 
From Lemma \ref{para-lem} (2) and (3), it is easy to see $h\in C^2(\C_{\theta})$ is a capillary convex function, together with Proposition \ref{equ pro}, it gives rise to a capillary convex body $L\in \mathcal{K}_\theta$ with $$h =\wt{h}_L|_{\C_\theta} ~~ \text{ in } ~~ \C_\theta.$$ From Lemma \ref{para-lem} (1), the position vector of $\p L\cap  {\RR^{n+1}_+}$ is given by
\begin{eqnarray*}
    X(\xi) &=& \n h(\xi) +h(\xi) T^{-1}(\xi)
   = \n \wt{h}_K(\xi)+\wt{h}_K(\xi) T^{-1}(\xi),\quad \xi\in \C_\theta.
\end{eqnarray*}The right hand side is exactly the position vector of $K$ restricting on $\C_\theta$, hence we conclude that $K=L\in \K_\theta.$ The last assertion follows from $h_K=\wt{h}_K|_{\C_\theta}$.
\end{proof}
	 
  Recall that the mixed discriminant $Q: (\RR^{n\times n})^n \to \RR$  is defined by
 \begin{eqnarray*}
     \det(\l_1A_1+\cdots +\l_m A_m)=\sum_{i_1,\ldots,i_n=1}^m \l_{i_1}\cdots \l_{i_n} Q(A_{i_1},\cdots, A_{i_n})
 \end{eqnarray*}for $m\in \NN$, $\l_1,\ldots, \l_m\geq 0$ and  the real symmetric  matrices $A_1, \ldots ,A_m\in \RR^{n\times n}$. If $(A_k)_{ij}$ denotes the $(i,j)$-element of the matrix $A_k$, then
		\begin{eqnarray*}
Q(A_{1}, \cdots, A_{n} )=\frac{1}{n!}\sum_{i_1,\cdots, i_n, j_1,\cdots, j_n}\delta_{j_{1}\cdots j_{n}}^{i_{1}\cdots i_{n}}(A_{1})_{i_{1}j_{1}}\cdots (A_{n})_{i_{n}j_{n}}.
\end{eqnarray*}It is clear that the mixed discriminant $Q$ is a multilinear and symmetric function (see, e.g., \cite[Chapter 4, Section 25.4]{BZ}).

For any $f\in C^2(\C_\theta)$, we denote $A[f]\coloneqq\n^2 f+f\s$. We introduce the multi-variable function $V: (C^{2}(\C_{\theta}))^{n+1}\rightarrow \RR$ given  by 
		\begin{eqnarray}\label{defn-V1}
			V(f_{1}, \cdots, f_{n+1})=\frac{1}{n+1}\int_{\C_{\theta}}f_{1}Q\left(A[f_{2}], \cdots, A[f_{n+1}]\right)d\s.
		\end{eqnarray}
		for $f_{i}\in C^{2}(\C_\theta), 1\leq i\leq n+1$.
	
\begin{prop}
		\label{lem symmetric}
		Let $f_{1}, f_{2}, \cdots, f_{n+1}$ be capillary functions, then $V$ is a symmetric and multi-linear function.
	\end{prop}
	\begin{proof}
		The multilinear property is obvious. By the symmetry of $Q$,  
	we just need to prove \begin{eqnarray}\label{sym}
			V(f_{1}, f_{2}, \cdots, f_{n+1})=V(f_{2}, f_{1}, \cdots, f_{n+1}).  
		\end{eqnarray}
	It is known that 
		\begin{eqnarray}\label{Q-def}
			Q(A[f_{2}], \cdots, A[f_{n+1}])=\sum_{i,j=1}^n Q^{ij}\cdot A[f_2]_{ij},
		\end{eqnarray}
		where
		\begin{eqnarray*}
			Q^{ij}\coloneqq\frac{\p Q}{\p A[f_2]_{ij}}
		\end{eqnarray*}
		satisfying
		\begin{eqnarray}\label{div free}
			\sum\limits_{i=1}^{n}\n_{i} Q^{ij} =0,\quad \text{for~all }~1\leq j\leq n.
		\end{eqnarray}
		Using \eqref{Q-def}, \eqref{div free}  and  integrating by parts twice, we get
		\begin{eqnarray*}
			V(f_{1}, f_{2},\cdots, f_{n+1})&=&\frac{1}{n+1}\int_{\C_{\theta}} f_{1}Q^{ij}\cdot (\n^{2}f_{2}(e_{i}, e_{j})+f_{2}\delta_{ij})d\s\\
			&=&\frac{1}{n+1}\int_{\C_{\theta}} f_{2}Q^{ij}\cdot (\n^{2}f_{1}(e_{i}, e_{j})+f_{1}\delta_{ij})d\s \\
			&&+\frac{1}{n+1}\int_{\partial\C_{\theta}} \left[f_{1} Q^{ij}\n_{e_{j}}f_{2}\cdot \<\mu, e_{i}\>-f_{2}Q^{ij}\n_{e_{i}}f_{1}\cdot\<\mu, e_{j}\>\right] ds.
		\end{eqnarray*}
		Here we use an orthonormal basis $\{e_i\}$ on $\C_\theta$ such that $e_n=\mu$ along $\p \C_\theta$.
	On $\p\C_\theta$, from \eqref{robin}, there holds
		\begin{eqnarray*}
			f_{1}Q^{ij}\n_{e_{j}}f_{2}\<\mu, e_{i}\>-f_{2}Q^{ij}\n_{e_{i}}f_{1}\<\mu, e_{j}\>=
					Q^{nn}\left(f_{1}\n_{\mu}f_{2}-f_{2}\n_{\mu}f_{1}\right)
			=0.
		\end{eqnarray*}
		Hence we complete the proof of \eqref{sym}. 
	\end{proof}

The following is a direct corollary of Proposition \ref{lem symmetric}. 
	\begin{cor}\label{Minkowski}
		For any capillary function $f$, there holds
		\begin{eqnarray}\label{minkowski formula}
			\int_{\C_{\theta}}f H_{k-1}(A[f])d\s=\int_{\C_{\theta}}\ell H_{k}(A[f])d\s, \quad 1\leq k\leq n,
		\end{eqnarray}
		where $H_{k}(A[f])$ is the $k$-th normalized elementary symmetric function evaluated at $A[f]$.   
	\end{cor}
	\begin{proof} Recall $\ell(\xi)=\sin^{2}\theta+\cos\theta \<\xi, e\>.$ Note 
	$$\n^2(\xi-\cos\theta e)=-(\xi-\cos\theta e).$$ Hence $A[\ell]=\n^2 \ell+\ell \s=\s$.
		From Proposition \ref{lem symmetric}, we have 
		\begin{eqnarray*}\label{prop2.8 apply1}
			V(\underbrace{f, \cdots, f}_{k~\text{copies}}, \ell,\cdots, \ell)=V(\ell, \underbrace{f,\cdots, f}_{k~\text{copies}}, \ell, \cdots, \ell).
		\end{eqnarray*}
		By definition, we have
			\begin{eqnarray*}
		V(\underbrace{f, \cdots, f}_{k~\text{copies}}, \ell ,\cdots, \ell ) &=&\frac{1}{n+1}\int_{\C_{\theta}}fQ(\underbrace{A[f], \cdots, A[f]}_{(k-1)~\text{copies}}, A[\ell ], \cdots,A[\ell ])d\s \\
		&=& \frac{1}{n+1}\int_{\C_{\theta}}f H_{k-1} (A[f])
		d\s,
		\end{eqnarray*}
		and \begin{eqnarray*}
		V(\ell, \underbrace{f,\cdots, f}_{k~\text{copies}},
		 \ell, \cdots, \ell)	 = \frac 1 {n+1} \int_{\C_{\theta}}\ell H_{k}(A[f])d\s.
		\end{eqnarray*}
The assertion follows from the previous Proposition.
\end{proof}

\begin{rem}When $f$ is the capillary support function of a convex capillary hypersurface, by using the fact $\tilde{\nu}: \S\to \C_\theta$ is a diffeomorphism, we can see that  identity \eqref{minkowski formula} is equivalent to the Minkowski-type formula in \cite[Proposition~2.6]{WWX2023}:
\begin{eqnarray*}
\int_\S H_{n-k+1}\<x, \nu\>dA=\int_{\S}(1+\cos\theta \<\nu,e\>)H_{n-k} dA,
\end{eqnarray*}
where $H_{j}$ is the normalized $j$-th mean curvature of $\S$.
\end{rem}  
 
Next, we interpret the classical mixed volume for capillary convex bodies in $\K_\theta$ as an integral over $\C_\theta$ by using the capillary support function.
In Proposition \ref{mink-sum}, we proved that, for $K_1,\ldots, K_m\in \K_\theta$, the Minkowski sum  $K=\sum\limits_{i=1}^m \l_iK_i\in \mathcal{K}_\theta$. Recall that the mixed volume $V(K_{i_{1}},\cdots, K_{i_{n+1}})$ is defined by
 \begin{eqnarray*}
			|K| &=&\sum\limits_{i_{1},\cdots, i_{n+1}=1}^{m}\lambda_{i_{1}}\cdots\lambda_{i_{n+1}}V(K_{i_{1}},\cdots, K_{i_{n+1}}).
		\end{eqnarray*}

 \begin{prop}\label{mixed-volume-eq} Let $K_1,\ldots, K_m\in \K_\theta$ and  $h_i$ be the  capillary support function of $K_i$. Then we have 
		\begin{eqnarray}\label{mixed volumes}
			V(K_{i_{1}}, \cdots, K_{i_{n+1}})=V(h_{i_{1}}, \cdots, h_{i_{n+1}})=\frac{1}{n+1}\int_{\C_{\theta}}h_{i_{1}}Q\left(A[h_{i_{2}}], \cdots, A[h_{i_{n+1}}]\right)d\s.
		\end{eqnarray}
\end{prop}	
	
	\begin{proof}		
  Since $\tilde{\nu}: \S\to \C_\theta$ is a diffeomorphism, we have
		\begin{eqnarray*}
			(n+1)|K|=\int_\S \<x,\nu\>dA=\int_{\C_\theta}h_{K}\det \left(A[h_{K}]\right)d\s.\end{eqnarray*} 
Recall that $h_K =\sum\limits_{i=1}^m \l_i h_i$ in $\C_\theta$, then 
		\begin{eqnarray*}
			(n+1)|K|&=&\int_{\C_\theta}h_{K}\det \left(A[h_{K}]\right)d\s
   \\&=&\sum\limits_{i_{1},\cdots, i_{n+1}=1}^{m} \lambda_{i_{1}}\cdots\lambda_{i_{n+1}}\left(\int_{\C_{\theta}}h_{i_{1}}Q(A[h_{i_{2}}], \cdots, A[h_{i_{n+1}}] )d\s\right).\end{eqnarray*}
   The assertion follows from the definition of the mixed volume.
	\end{proof}

We collect the following basic properties for $V(h_1,\cdots, h_{n+1})$.
\begin{prop}\label{prop of mixed vol}Let  $K_1,\ldots, K_{n+1}\in \K_\theta$ and $h_i$ be the  capillary support function of $K_i$. Then    
\begin{enumerate}
    \item $V(h_1,h_1,\cdots, h_1)=|K_1|$.

   \item $V(h_1,h_2,\cdots, h_{n+1})$ is symmetric and multilinear in its arguments.

   \item $V(h_1,h_2,\cdots, h_{n+1})$ is invariant under the horizontal translation of $K_j$ in $\ol{\RR^{n+1}_+}$.

   \item $V(h_1,h_2,\cdots, h_{n+1})\geq 0$, if $h_i\ge 0$.
\end{enumerate}
\end{prop}

 Finally, we demonstrate that the quermassintegral $\V_{k,\theta}(\wh{\S})$, which was introduced in \cite{WWX2023}, is a specific case of the mixed volume.
To see this, we need the following reinterpretation of $\V_{k,\theta}(\wh{\S})$.

\begin{lem}\label{prop-two forms V-k} Let $\wh{\S}\in \K_\theta$. 
  Then    
  \begin{eqnarray}\label{eq_last}
  \V_{k+1,\theta} (\widehat  \S) =\frac 1 {n+1}
    \int_{\S}(1+\cos\theta \<\nu,e\>) H_k dA ,\quad 0\leq k\le n.
\end{eqnarray}
  	  \end{lem}
    \begin{proof} 
   By integration by parts, we know
   \begin{eqnarray*}
     0=\int_{\wh\S} {\rm div}(e) dx=  |\widehat{\p \S}|+\int_\S \<\nu, e\>dA.
   \end{eqnarray*}
 The assertion for  $k=0$ follows.

    Next, we consider the case $1\le k\le n$. Let $B=(B_{ij})_{1\leq i,j\leq n}$ and $\widehat{B}=(\wh{B}_{\a\b})_{1\leq \a,\b\leq n-1}$ be the second fundamental form of $\S\subset\RR^{n+1}_+$ and $\p\S \subset \p \ov {\RR_+^{n+1}}$ respectively. Let $\{e_i\}_{i=1}^n$ be an orthonormal basis on $\S$. It is known (cf. \cite[Proposition 2.4 (1) (2)]{WWX2023}) that $\mu\coloneqq e_n$ is the principal direction of $\S$ along $\p\S$ and $B_{\a\b}=\sin\theta \widehat{B}_{\a\b}$. 
    We notice that  on $\S$, 
    \begin{eqnarray}\label{xeq1}
        \s^{ij}_k\<\n^\S_ie^T, e_j\>=-\s^{ij}_kB_{ij}\<\nu, e\>= -k\s_k\<\nu, e\>,
    \end{eqnarray}
    where $\s_k^{ij}=\frac{\p \s_k}{\p B_{ij}}$.
    Integrating \eqref{xeq1} over $\S$ and integrating by parts, we see that
    \begin{eqnarray*}
        \int_\S k\s_k\<\nu, e\>dA=-\int_{\p\S}\s^{\mu\mu}_k \<e,\mu\>ds=-\sin^k\theta\int_{\p\S}\s_{k-1}^{\p\S}ds,
    \end{eqnarray*}
    where $\s_{k-1}^{\p\S}$ is $(k-1)$-th mean curvature of $\p\S\subset \p\RR_+^{n+1}$. 
    It follows that
      \begin{eqnarray*}
        \int_\S H_k\<\nu, e\>dA=-\frac{1}{n}\sin^k\theta\int_{\p\S}H^{\p\S}_{k-1}ds.
    \end{eqnarray*}
    This completes the proof.
  \end{proof}

	\begin{prop} \label{prop V and V-k}
Let $\Sigma\subset \ov{\RR^{n+1}_+}$ be a convex capillary hypersurface and 
$h$  its capillary support function. Then for  $0\leq k\leq n+1$,
		\begin {eqnarray}\label{V and V-k}
			V(\underbrace{\widehat{\Sigma}, \cdots, \widehat{\Sigma}}_{k ~\rm{ copies}} ,\widehat{\C_\theta},\cdots, \widehat{\C_\theta})=V(\underbrace{h, \cdots, h}_{k ~\rm{ copies}} ,\ell,\cdots, \ell)= \mathcal{V}_{n-k+1,\theta}(\widehat{\Sigma}).
		\end{eqnarray}
	\end{prop}
	\begin{proof}
 If $k=n+1$, \eqref{V and V-k} follows directly from Proposition \ref{prop of mixed vol} (1). We only need to show the case  $0\leq k\leq n$.
		From Proposition \ref{equ pro}, we know that $h$ is a capillary convex function, then 		
  \begin {eqnarray*} \label{eq2.2}
		V(\underbrace{h, \cdots, h}_{k \text{ copies}} ,\ell,\cdots, \ell) &=&\frac 1 
		{n+1}  \int _{\C_{\theta} } \ell Q (\underbrace{A[h], \cdots , A[h]}_{k \text{ copies}}, A[\ell], \cdots , A[\ell]) d\s \\&=& \frac 1 
		{n+1}  \int _{\C_{\theta} } \ell  H_k(A[h]) d\s \\
		&=&\frac{1}{n+1}\int _{\S}\left(1+\cos\theta\<\nu, e\>\right)H_{n-k}\,dA.\end{eqnarray*}
  The last equality follows from the area formula by using the diffeomorphism $\tilde{\nu}: \S\to \C_\theta$.
Together with \eqref{eq_last}, we complete the proof of \eqref{V and V-k}. 
\end{proof}

As a direct consequence of the above reinterpretation, we have the following Steiner-type formula.
\begin{prop} Let $\wh{\S}\in \K_\theta$. Then for any $t>0$, there hold
\begin{eqnarray*}  \label{Steiner2} 
  &&|\widehat {\S}+t\widehat{\C_\theta}| = \sum_{k=0}^{n+1} \binom{n+1}{k}  t^k \V_{k,\theta}(\widehat\S).
  \end{eqnarray*}
  \end{prop}
\begin{rem}
 We make further comments on the Minkowski sum 
$\widehat {\S}+t\widehat{\C_\theta}$. One can define the following suitable parallel hypersurface of $\S$. For any $t>0$,  let
  \begin{eqnarray} \label{parallel capillary}
  \Sigma_t\coloneqq \{x+t(\nu+\cos \theta e) \, |  x\in \Sigma\}.
  \end{eqnarray}
One observes that 
\begin{eqnarray*}\label{Sigma_t}\wh{\Sigma_t}=\widehat {\S}+t\widehat{\C_\theta}.\end{eqnarray*}
This is because a point $y\in \C_\theta$ is given by $y=\nu(y)+\cos \theta e$.
From this observation, one sees the unit outward normal $\nu_t$ of $\S_t$ is equal to $\nu$ and the volume form $dA_t$ of $\S_t$ is given by 
$$dA_t=\sum\limits_{k=0}^nt^k\binom{n} {k} H_k dA.$$ Therefore, by virtue of \eqref{eq_last}, we get that
\begin{eqnarray*}
 		\V_{1,\theta}(\widehat {\S_t})&=& \frac{1}{n+1}\int_{\S_t}(1+\cos\theta \<\nu_t,e\>)  dA_{t} 
 		\\&=&\frac{1}{n+1}\sum_{k=0}^n  \int_{\S}(1+\cos\theta \<\nu,e\>)t^k\binom{n} {k} H_k dA 
 		\\&=& \sum_{k=0}^n \binom{n}{k}  t^k \V_{k+1,\theta}(\widehat\S).
 	\end{eqnarray*}
We remark that recently the parallel hypersurfaces were used in \cite{JWXZ} to prove the Heintze-Karcher inequality for capillary hypersurfaces. 
\end{rem}

\section{The Alexandrov-Fenchel inequalities for mixed volume}\label{sec3}

With preparation in the previous section, we prove Theorem \ref{thm1} in this section, adapting the idea of Shenfeld-van Handel in \cite{Shenfeld, SV22}.  For the reader’s convenience, we include \cite[Lemma 1.4]{Shenfeld} below, as it will be crucially used in the subsequent analysis.
\begin{lem}[\cite{Shenfeld}]\label{lemma3.1} Let $A$ be a symmetric matrix. Then the following conditions are equivalent:
\begin{enumerate}
    \item $\<x,Ay\>^2\geq \<x,Ax\>\<y, Ay\>$ for all $x,y$ such that $\<y, Ay\>\geq 0$.

    \item The positive eigenspace of $A$ has dimension at most one.
    \end{enumerate}
 The conclusion remains valid if $A$ is a self-adjoint operator on a Hilbert space with a discrete spectrum, provided the vectors $x,y$ are chosen in the domain of $A$.
\end{lem}
We now proceed to prove the following theorem.
\begin{thm}\label{thm1-old}
 Let $f_{1}, \cdots, f_{n}\in C^2(\C_{\theta})$ be capillary convex functions on $\C_{\theta}$ and  at least one of $f_{k}~(2\leq k\leq n)$ is positive. Then for any capillary function $f\in C^2(\C_\theta)$, there holds
	\begin{eqnarray}\label{function AF}
		V^{2}(f, f_{1}, f_{2},\cdots, f_{n})\geq V(f, f, f_{2},\cdots, f_{n}) V(f_{1}, f_{1}, f_{2},\cdots, f_{n}).
	\end{eqnarray}
	Equality holds if and only if $f=af_{1}+\sum\limits_{i=1}^{n}a_{i}\<\cdot, E_{i}\>$ for some constants $a, a_{i} \in \RR$ and $\{E_i\}_{i=1}^n$ being the horizontal coordinate unit vectors of $\ol{\RR^{n+1}_+}$.
 \end{thm}

\begin{proof}
 Without loss of generality, we assume that $f_{2}$ is positive on $\C_{\theta}$.
Define $\omega$ to be the following area measure on $\C_\theta$
	$$d\omega\coloneqq\frac{1}{n+1}\frac{Q(A[f_{2}], A[f_{2}],\cdots,A[f_{n}])}{f_{2}}d\s.$$
Let $L^{2}(\C_{\theta}, \omega)$ be the Hilbert space associated with the inner product $\< , \>_{L^{2}(\omega)}$  given by 
	\begin{eqnarray*}
		\<f, g\>_{L^{2}(\omega)}\coloneqq\int_{\C_{\theta}}f\cdot g d\omega, \quad \forall f, g\in L^{2}(\C_{\theta}, \omega). 
	\end{eqnarray*}
$W^{2,2}(\C_{\theta}, \omega)$ is similarly defined. 
 Define the following operator as in \cite{Shenfeld}:
 \begin{eqnarray*}\label{operator}
 &&\mathcal{A}: 
	\text{dom}(\mathcal{A})\subset L^{2}(\C_{\theta}, \omega)\rightarrow L^{2}(\C_\theta, \omega),\\
	&&
		\mathcal{A}(f)=\frac{f_{2}Q(A[f], A[f_{2}], \cdots, A[f_{n}])}{Q(A[f_{2}], A[f_{2}],\cdots, A[f_{n}])},
	\end{eqnarray*}
	where 
	$\text{dom}(\mathcal{A})$ is given by 
	\begin{eqnarray*}
		\text{dom}(\mathcal{A})\coloneqq \left\{f\in W^{2,2}(\C_{\theta}, \omega): \n_{\mu}f=\cot\theta f~ \text{on}~\partial\C_{\theta}\right\}.
	\end{eqnarray*}
	Note that all the above objects are well-defined, due to $f_{2}>0$ and   $f_{1}, \cdots, f_{n}$ are capillary convex functions on $\C_{\theta}$. 
	
	By the very definition of $\mathcal{A}$,  \eqref{function AF} is equivalent to	\begin{eqnarray}\label{key-ine}
		\left(\<f, \mathcal{A} f_{1}\>_{L^{2}(\omega)}\right)^{2}\geq \<f, \mathcal{A}f\>_{L^{2}(\omega)}\<f_{1}, \mathcal{A}f_{1}\>_{L^{2}(\omega)},
	\end{eqnarray}
	for functions $f, \{f_{i}\}_{i=1}^{n}$ that satisfy the assumptions in Theorem \ref{thm1}. By Lemma \ref{lemma3.1}, we only need to prove that the positive eigenspace of $\mathcal{A}$ has dimension at most one, which we are going to do. 
	
	Since $f_{1}, \cdots, f_{n}$ are capillary convex functions, we know that $\mathcal{A}$ is a uniformly elliptic operator with a Robin boundary condition.  Proposition \ref{lem symmetric} implies 
	\begin{eqnarray*}
		\<f, \mathcal{A}g\>_{L^{2}(\omega)}=\<g, \mathcal{A}f\>_{L^{2}(\omega)}, \quad \forall f, g\in \text{dom}(\mathcal{A}),
	\end{eqnarray*}
	which means that $\mathcal{A}$ is a self-adjoint operator on $\text{dom}(\mathcal{A})$.  Consider the following eigenvalue problem
\begin{eqnarray}\label{eigenvalue problem}
		\begin{array}{rcll}
			\mathcal{A}f&=&  \lambda f ,& \quad    \hbox{ in } \C_{\theta},\\
			\n_\mu f&=&\cot\theta f, &\quad  \hbox{ on } \p \C_\theta.
		\end{array}
	\end{eqnarray}
The classical spectral theory of compact, self-adjoint operators implies that the eigenvalue problem \eqref{eigenvalue problem} admits a sequence of real eigenvalues $\l_1\geq \l_2\geq \cdots$,  with $\lambda_k\rightarrow -\infty$, its first eigenvalue $\lambda_{1}$ is simple and its first eigenfunction has the same sign everywhere. 
  It is trivial to check that $f_2$ satisfies \eqref{eigenvalue problem} with $\lambda=1$. Since  $f_2>0$,  $f_2$ is the first eigenfunction with the first eigenvalue $\lambda_1=1$.

	On the other hand, by using Alexandrov's mixed discriminant inequality (see e.g. \cite[Theorem 5.5.4]{Sch})
\begin{eqnarray}\label{alex-ineq}
     \left(Q(A[{g}], A[f_{2}], \cdots, A[f_{n}])\right)^{2}\ge Q(A[f_{2}], A[f_{2}],\cdots, A[f_{n}])\cdot Q(A[{g}], A[{g}], A[f_{2}], \cdots, A[f_{n}]),~~
 \end{eqnarray}
 together with Proposition \ref{lem symmetric}, for all function ${g}\in \text{dom}({\mathcal{A}})$ we have  
\begin{eqnarray}\label{eigen-est}
		\<\mathcal{A}{g}, \mathcal{A}{g}\>_{L^{2}(\omega)}&=&\frac{1}{n+1}\int_{\C_{\theta}}\frac{f_{2}\left(Q(A[{g}], A[f_{2}], \cdots, A[f_{n}])\right)^{2}}{Q(A[f_{2}], A[f_{2}],\cdots, A[f_{n}])}d\s\nonumber\\
		&\geq&\frac{1}{n+1}\int_{\C_{\theta}}f_{2}Q(A[{g}], A[{g}], A[f_{2}], \cdots, A[f_{n}])d\s\nonumber\\
		&=&\frac{1}{n+1}\int_{\C_{\theta}}{g} Q(A[{g}], A[f_{2}], \cdots, A[f_{n}])d\s \nonumber\\&=& \<{g}, \mathcal{A}{g}\>_{L^{2}(\omega)}.
	\end{eqnarray}
	Therefore, if $\lambda$ is an eigenvalue of the operator $\mathcal{A}$, the above inequality implies that $\lambda^{2}\geq \lambda$, which means $\lambda\geq 1$ or $\lambda\leq 0$. 
	Consequently,  the positive eigenspace of $\mathcal{A}$ is of one dimension and is spanned by $f_{2}$. In view of Lemma \ref{lemma3.1}, we complete the proof of \eqref{key-ine}. 

Next, we characterize equality. If  equality holds in \eqref{function AF}, then also equality holds in \eqref{key-ine},  that is
	\begin{eqnarray}\label{equality AF1}
	\left(\<f, \mathcal{A} f_{1}\>_{L^{2}(\omega)}\right)^{2}= \<f, \mathcal{A}f\>_{L^{2}(\omega)}\<f_{1}, \mathcal{A}f_{1}\>_{L^{2}(\omega)}.	\end{eqnarray}
From \cite[Lemma~2.9 and its proof]{Shenfeld},   equality holds in \eqref{equality AF1} if and only if $\widetilde{f}\coloneqq f-af_{1}\in \text{Ker}(\mathcal{A})$ for some constant $a\in \RR$.

For $g=\tilde{f}$, equality holds in \eqref{eigen-est} and in turn in Alexandrov's mixed discriminant inequality \eqref{alex-ineq}. It follows from \cite[Theorem 5.5.4]{Sch} that
$A[\tilde{f}]=c A[f_2]$ for some function $c: \C_\theta\to \RR$.
Since $\tilde{f}\in \text{Ker}(\mathcal{A})$, we have
\begin{eqnarray*}
   0=\frac{f_{2}Q(A[\tilde{f}-cf_2], A[f_{2}], \cdots, A[f_{n}])}{Q(A[f_{2}], A[f_{2}],\cdots, A[f_{n}])}=\mathcal{A}\tilde{f}-c=-c. 
\end{eqnarray*}
This implies $A[\tilde{f}]=0$ and 
hence $\tilde f$ is a linear function. The fact  that  $\tilde{f}$ is a capillary function implies  that 
	$\tilde{f}=\sum\limits_{i=1}^{n}a_{i}\<\cdot, E_{i}\>$ for some constants $a_i \in\RR$, $i=1,\cdots, n$. Therefore, equality in \eqref{function AF} holds if and only if $f=af_{1}+\sum\limits_{i=1}^{n}a_{i}\<\cdot, E_{i}\>.$
\end{proof}

\begin{proof}[\textbf{Proof of Theorem \ref{thm1}}]

Taking $f$ and $f_i$  $(i=1,\cdots, n)$ to be the capillary support functions of $\widehat{\S}_1$ and  $\widehat{\S}_{i+1}(i=1,\cdots, n)$, in Theorem \ref{thm1}. By a horizontal translation, we may assume $f_i>0$. Applying Theorem \ref{thm1-old} and taking into account of Proposition \ref{mixed-volume-eq}, we obtain the desired result.
\end{proof}

\medskip

By Theorem \ref{thm1}, following the same argument as in \cite[Section~7.4 and (7.63)]{Sch}, we obtain the following inequality for capillary convex bodies.
\begin{cor}\label{coro-AF}

	 Let $m\in \{1, \cdots, n+1\}$ and  $K_{0}, K_{1},  K_{m+1},\cdots, K_{n+1}\in\mathcal{K}_{\theta}$, and $h_j$ be the capillary support function of $K_j$.  Let
	\begin{eqnarray}\label{Wi}
V_{(i),\theta}\coloneqq V(\underbrace{h_{0},\cdots, h_{0}}_{(m-i)~\rm{copies}}, \underbrace{h_{1},\cdots, h_{1}}_{i~\rm{copies}}, h_{m+1},\cdots, h_{n+1}),\quad \text{for}\quad i=0,\cdots, m.\end{eqnarray}
Then $V_{(i),\theta}$ satisfies 
	\begin{eqnarray}\label{W-AF}
		V_{(j), \theta}^{k-i}\geq V_{(i), \theta}^{k-j}V_{(k), \theta }^{j-i}, ~~ ~~ 0\leq i<j<k\leq m\leq n+1.
	\end{eqnarray}
 Equality holds if and only if $h_{0}=ah_{1}+\sum\limits_{i=1}^{n}a_{i}\<\xi, E_{i}\>$ for some constants $a$ and $a_{i}$, $1\leq i\leq n$.	
\end{cor}
As a consequence, we get Theorem \ref{thm2}.
\begin{proof}[\textbf{Proof of Theorem \ref{thm2}}]

By choosing $k=m=n+1$, $K_0=\wh{\S}$ and $K_{1}=\widehat{\C_{\theta}}$ in  Corollary \ref{coro-AF}, together with the fact that
$V_{(n+1),\theta}=|\widehat{\C_{\theta}}|=\bt$ and Proposition \ref{prop V and V-k}, we obtain the desired result. 
\end{proof}

\

\noindent\textit{Acknowledgment}: We would like to thank the referee for careful reading to improve the context of the paper. X.M. was partially supported by the National Key R$\&$D Program of China 2020YFA0712800  and the Postdoctoral Fellowship Program of CPSF under Grant Number GZC20240052. L.W. was partially supported by CRM De Giorgi of Scuola Normale Superiore and PRIN Project 2022E9CF89 of University of Pisa. C.X. was partially supported by the NSFC (Grant No. 12271449, 12126102) and the Natural Science Foundation of Fujian Province of China (Grant No. 2024J011008).

\ 

\printbibliography 
\end{document}